\documentclass[english]{article}
\usepackage{lmodern}

\usepackage[T1]{fontenc}
\usepackage[latin9]{inputenc}
\usepackage{amsmath}
\usepackage{amsthm}

\makeatletter
\theoremstyle{plain}
\newtheorem{thm}{\protect\theoremname}
\theoremstyle{plain}
\newtheorem{lem}[thm]{\protect\lemmaname}
\ifx\proof\undefined
\newenvironment{proof}[1][\protect\proofname]{\par
	\normalfont\topsep6\p@\@plus6\p@\relax
	\trivlist
	\itemindent\parindent
	\item[\hskip\labelsep\scshape #1]\ignorespaces
}{%
	\endtrivlist\@endpefalse
}
\providecommand{\proofname}{Proof}
\fi
\theoremstyle{plain}
\newtheorem{cor}[thm]{\protect\corollaryname}
\newcommand{\lyxaddress}[1]{
	\par {\raggedright #1
	\vspace{1.4em}
	\noindent\par}
}

\@ifundefined{date}{}{\date{}}
\makeatother

\usepackage{babel}
\providecommand{\corollaryname}{Corollary}
\providecommand{\lemmaname}{Lemma}
\providecommand{\theoremname}{Theorem}

\begin{document}
\title{Gap theorems in Yang-Mills theory for complete four-dimensional manifolds
with positive Yamabe constant}
\author{Matheus Vieira}
\maketitle
\begin{abstract}
In this paper we prove gap theorems in Yang-Mills theory for complete
four-dimensional manifolds with positive Yamabe constant. We extend
the results of Gursky-Kelleher-Streets to complete manifolds. We also
describe the equality in the gap theorem in terms of the basic instanton,
which is interesting even for compact manifolds.
\end{abstract}

\section{Introduction}

Consider a complete four-dimensional Riemannian manifold $X$, a Riemannian
vector bundle $V$ on the manifold $X$ and a metric connection $A$
on the bundle $V$ with curvature $F$. We define the Yang-Mills energy
of the connection $A$ by
\[
YM\left(A\right)=\int_{X}\left|F\right|^{2}.
\]
We say that the connection $A$ is Yang-Mills if the divergence of
the curvature $F$ is zero ($d^{*}F=0$), which is the Euler-Lagrange
equation of the Yang-Mills functional $YM$. We say that the connection
$A$ is an instanton if the curvature $F$ is anti-self-dual ($F^{+}=0$)
or self-dual ($F^{-}=0$). Here $F^{\pm}=1/2\left(F\pm*F\right)$.
Instantons are an important class of Yang-Mills connections. In \cite{AHDM1978}
and \cite{BPST1975} Atiyah-Hitchin-Drinfeld-Manin and Belavin-Polyakov-Schwartz-Tyupkin
found an explicit $SU\left(2\right)$ instanton, known as the basic
instanton. In \cite{SSU1989} Sibner-Sibner-Uhlenbeck proved the existence
of $SU\left(2\right)$ Yang-Mills connections which are not instantons
(see also Bor \cite{B1992}, Parker \cite{P1992}, Sadun-Segert \cite{SS1992}).
We remark that for bigger groups such as $SU\left(4\right)$ we can
easily find Yang-Mills connections which are not instantons (see for
example Section 2.1 in \cite{VN2008}).

In \cite{BLS1979} and \cite{BL1981} Bourguignon-Lawson-Simons proved
the following $L^{\infty}$ gap theorem: if a Yang-Mills connection
with curvature $F$ on the sphere $S^{4}$ satisfies $F^{\pm}\neq0$,
then the $L^{\infty}$ norm of $F^{\pm}$ is bounded below by an explicit
constant (see also Shen \cite{S1982}). In \cite{M1982} MinOo proved
the following $L^{2}$ gap theorem: if a Yang-Mills connection with
curvature $F$ on the sphere $S^{4}$ satisfies $F^{\pm}\neq0$, then
the $L^{2}$ norm of $F^{\pm}$ is bounded below by an explicit constant
(see also Dodziuk-MinOo \cite{DM1982}, Parker \cite{P1982}). In
\cite{F2016} Feehan proved a $L^{2}$ gap of $F^{\pm}$ for compact
manifolds with a good metric (the extension of this result to general
metrics is an open problem). In \cite{GKS2018} Gursky-Kelleher-Streets
proved a sharp $L^{2}$ gap of $F^{\pm}$ for compact manifolds with
positive Yamabe constant. In \cite{V2019} we proved a $L^{\infty}$
type gap of $F^{\pm}$ for complete manifolds with a weighted Poincare
inequality. For higher-dimensional manifolds gap theorems for the
curvature $F$ were proved by Price \cite{P1983}, Gerhardt \cite{G2010},
Zhou \cite{Z2015}-\cite{Z2016}, Feehan \cite{F2017}-\cite{F2024},
among others. For four-dimensional manifolds gap theorems for the
self-dual and anti-self-dual parts of the curvature $F$, that is
$F^{+}$ and $F^{-}$, are more interesting.

In the next result we extend a gap theorem of Gursky-Kelleher-Streets
\cite{GKS2018} from compact manifolds to complete manifolds with
at most Euclidean volume growth. We denote the scalar curvature and
Weyl curvature by $S$ and $W$, respectively, and we denote $W^{\pm}=1/2\left(W\pm*W\right)$.
In the same way as Schoen-Yau \cite{SY1988}, we define the Yamabe
constant $C_{Y}$ of a complete (compact or noncompact) four-dimensional
manifold $X$ by
\[
C_{Y}=\inf_{\phi\in C_{c}^{\infty}\left(X\right)}\frac{6\int_{X}\left|\nabla\phi\right|^{2}+\int_{X}S\phi^{2}}{\left(\int_{X}\phi^{4}\right)^{1/2}}.
\]
We denote by $\gamma$ a constant depending only on $G$ (see Lemma
\ref{lem:bochner}).
\begin{thm}
\label{thm:inequality}Consider a complete four-dimensional Riemannian
manifold $X$ with volume growth $vol\left(B_{R}\right)\leq const\cdot R^{4}$
and positive Yamabe constant $C_{Y}>0$. Consider a Yang-Mills connection
$A$ with curvature $F$ and structure group $G\subset O\left(N\right)$
on the manifold $X$, where $N\geq3$. If $F^{+}\neq0$, then
\[
3\gamma\left|F^{+}\right|_{L^{2}}+2\sqrt{6}\left|W^{+}\right|_{L^{2}}\geq C_{Y}.
\]
The same statement is true replacing $F^{+}$ and $W^{+}$ by $F^{-}$
and $W^{-}$, respectively.
\end{thm}
In the next result we study the case of equality in the above inequality.
In this case, by Lemma \ref{lem:bochner} and the proof of Theorem
\ref{thm:equality}, we see that anti-self-dual Weyl curvature is
a natural assumption. In \cite{GKS2018} Gursky-Kelleher-Streets proved
that if the manifold is compact and equality holds, then $F^{+}$
is parallel in a certain conformal metric. We prove that if the manifold
is complete (compact or noncompact) and equality holds, then in a
certain way $F^{+}$ is a multiple of the basic instanton. This is
interesting even for compact manifolds.
\begin{thm}
\label{thm:equality}Consider a complete four-dimensional Riemannian
manifold $X$ with volume growth $vol\left(B_{R}\right)\leq const\cdot R^{4}$,
positive Yamabe constant $C_{Y}>0$ and anti-self-dual Weyl curvature
($W^{+}=0$). Consider a Yang-Mills connection $A$ with curvature
$F$ and structure group $G\subset O\left(N\right)$ on the manifold
$X$, where $N\geq4$. If
\[
3\gamma\left|F^{+}\right|_{L^{2}}=C_{Y},
\]
then at each point, in normal coordinates, the triple $F_{12}^{+}$,
$F_{13}^{+}$, $F_{14}^{+}$ is simultaneously orthogonally equivalent
to a triple $a_{1}\boldsymbol{i}$, $a_{2}\boldsymbol{j}$, $a_{3}\boldsymbol{k}$
of multiples of a basis of $su\left(2\right)$ embedded into $so\left(N\right)$,
and the absolute values of $a_{1}$, $a_{2}$, $a_{3}$ are equal.
The same statement is true replacing $F^{+}$ and $W^{+}$ by $F^{-}$
and $W^{-}$, respectively.
\end{thm}
The extension of Theorem \ref{thm:equality} from compact manifolds
to complete manifolds requires a new approach. The proof in \cite{GKS2018}
is based on solving a modified Yamabe problem on compact manifolds,
which is not possible for general complete manifolds. Our proof is
based on a sharp algebraic inequality (Lemma \ref{lem:linearalgebra})
and a sharp Bochner type formula (Lemma \ref{lem:bochner}).

The paper is organized as follows. In Section 2 we check that Theorem
\ref{thm:equality} is sharp in the basic instanton. In Section 3
we prove an algebraic inequality. In Section 4 we prove a Bochner
type formula. In Section 5 we prove Theorem \ref{thm:inequality}
and Theorem \ref{thm:equality} and we apply these results to the
sphere $S^{4}$, the Euclidean space $R^{4}$ and the cylinder $S^{3}\times R$.

\section{Basic instanton}

In this section we check that Theorem \ref{thm:equality} is sharp
in the basic instanton.

First we introduce a basis of $su\left(2\right)$ embedded into $so\left(N\right)$,
where $N\geq4$. Take matrices $\boldsymbol{i}$, $\boldsymbol{j}$,
$\boldsymbol{k}$ in $so\left(4\right)$ satisfying the quaternion
multiplication table and identify $\boldsymbol{i}$, $\boldsymbol{j}$,
$\boldsymbol{k}$ with the matrices
\[
diag\left(\boldsymbol{i},0_{N-4}\right),\,\,\,\,\,diag\left(\boldsymbol{j},0_{N-4}\right),\,\,\,\,\,diag\left(\boldsymbol{k},0_{N-4}\right),
\]
respectively. We see that $\boldsymbol{i}$, $\boldsymbol{j}$, $\boldsymbol{k}$
are matrices in $so\left(N\right)$ satisfying a quaternion type multiplication
table:
\[
\boldsymbol{i}^{2}=\boldsymbol{j}^{2}=\boldsymbol{k}^{2}=\boldsymbol{i}\boldsymbol{j}\boldsymbol{k}=-diag\left(I_{4},0_{N-4}\right).
\]
This is a basis of $su\left(2\right)$ embedded into $so\left(N\right)$.
We define the inner product of $so\left(N\right)$ as follows
\[
\left\langle A,B\right\rangle _{so\left(N\right)}=-c\cdot tr\left(AB\right),\,\,\,\,\,c>0.
\]
For example, for a basis $\boldsymbol{i}$, $\boldsymbol{j}$, $\boldsymbol{k}$
of $su\left(2\right)$ embedded into $so\left(N\right)$ we see that
each basis vector has norm $2\sqrt{c}$. We remark that the value
of the constant $c$ varies in the gauge theory literature.

Next we write the curvature of the basic instanton in a basis $\boldsymbol{i}$,
$\boldsymbol{j}$, $\boldsymbol{k}$ of $su\left(2\right)$. We follow
Section 3.4.1 in Donaldson-Kronheimer \cite{DK1997} (see also \cite{AHDM1978},
\cite{BPST1975}). Take the connection $A$ on the Euclidean space
$R^{4}$ given by
\[
A=\left(1+\left|x\right|^{2}\right)^{-1}\left(\theta_{1}\boldsymbol{i}+\theta_{2}\boldsymbol{j}+\theta_{3}\boldsymbol{k}\right),
\]
where
\[
\theta_{1}=x_{1}dx_{2}-x_{2}dx_{1}-x_{3}dx_{4}+x_{4}dx_{3},
\]
\[
\theta_{2}=x_{1}dx_{3}-x_{3}dx_{1}+x_{2}dx_{4}-x_{4}dx_{2},
\]
\[
\theta_{3}=x_{1}dx_{4}-x_{4}dx_{1}-x_{2}dx_{3}+x_{3}dx_{2}.
\]
By a classic calculation, we see that the curvature $F$ of this connection
is
\[
F=2\left(1+\left|x\right|^{2}\right)^{-2}\left\{ \left(dx_{12}-dx_{34}\right)\boldsymbol{i}+\left(dx_{13}+dx_{24}\right)\boldsymbol{j}+\left(dx_{14}-dx_{23}\right)\boldsymbol{k}\right\} ,
\]
where $dx_{ij}=dx_{i}\wedge dx_{j}$. We see that the curvature $F$
is anti-self-dual, so we conclude that the connection $A$ is an instanton,
known as the basic instanton.

Finally we check that Theorem \ref{thm:equality} is sharp in the
basic instanton. The theorem is sharp in the integral sense, that
is $3\gamma\left|F\right|_{L^{2}}=C_{Y}$ (see \cite{GKS2018} using
the inner products of Lemma \ref{lem:bochner}). The theorem is also
sharp in the pointwise sense because at each point $x$ the triple
$F_{12}$, $F_{13}$, $F_{14}$ is equal to a triple $a\boldsymbol{i}$,
$a\boldsymbol{j}$, $a\boldsymbol{k}$ of multiples of a basis of
$su\left(2\right)$, where $a=2\left(1+\left|x\right|^{2}\right)^{-2}$.

\section{Algebraic inequality}

In this section we prove an algebraic inequality. Part of the proof
is inspired by Bourguignon-Lawson (Proposition 5.6 in \cite{BL1981}).
The description of equality in the next result is very important in
the proof of Theorem \ref{thm:equality}.
\begin{lem}
\label{lem:linearalgebra}Consider a two-form $\omega$ on $R^{4}$
with values in $so\left(N\right)$, where $N\geq3$. Suppose that
the form $\omega$ is self-dual ($*\omega=\omega$) or anti-self-dual
($*\omega=-\omega$). Then
\[
\left|\sum_{i,j,k=1}^{4}\left\langle \left[\omega_{ij},\omega_{jk}\right],\omega_{ki}\right\rangle \right|\leq\gamma\left|\omega\right|^{3}.
\]
Also, if $\omega\neq0$ and equality holds (in the above inequality),
then:

(i) The $\omega_{ij}$ ($i\neq j$) have the same norm, namely $\left|\omega_{ij}\right|=\left(1/\sqrt{6}\right)\left|\omega\right|$.

(ii) For $N=3$ the triple $\omega_{12}$, $\omega_{13}$, $\omega_{14}$
is an orthogonal basis of $so\left(3\right)$.

(iii) For $N\geq4$ there exist constants $a_{1}$, $a_{2}$, $a_{3}$
and a basis $\boldsymbol{i}$, $\boldsymbol{j}$, $\boldsymbol{k}$
of $su\left(2\right)$ embedded into $so\left(N\right)$ (see Section
2) such that the triple $\omega_{12}$, $\omega_{13}$, $\omega_{14}$
is simultaneously orthogonally equivalent to the triple $a_{1}\boldsymbol{i}$,
$a_{2}\boldsymbol{j}$, $a_{3}\boldsymbol{k}$.

Here we use the inner products
\[
\left|\omega\right|^{2}=\sum_{1\leq i<j\leq4}\left|\omega_{ij}\right|^{2},\,\,\,\,\,\left\langle A,B\right\rangle _{so\left(N\right)}=-c\cdot tr\left(AB\right),\,\,\,\,\,c>0,
\]
and we denote
\[
\gamma=\begin{cases}
4/\sqrt{12c}, & N=3,\\
4/\sqrt{6c}, & N\geq4.
\end{cases}
\]
\end{lem}
\begin{proof}
We prove the result for self-dual forms. The proof for anti-self-dual
forms is similar.

Using the fact that the form $\omega$ is self-dual and the inner
product of $so\left(N\right)$ is ad-invariant, we have
\[
\sum_{i,j,k=1}^{4}\left\langle \left[\omega_{ij},\omega_{jk}\right],\omega_{ki}\right\rangle =24\left\langle \left[\omega_{12},\omega_{23}\right],\omega_{31}\right\rangle .
\]
Using an inequality of Bourguignon-Lawson (Lemma 2.30 and Note 2.31
in \cite{BL1981}) and noting that B-L use the inner product $\left\langle A,B\right\rangle _{so\left(N\right)}=-\left(1/2\right)tr\left(AB\right)$,
we have
\[
\left|\left[\omega_{12},\omega_{23}\right]\right|\leq\tilde{c}\left|\omega_{12}\right|\left|\omega_{23}\right|,
\]
where
\[
\tilde{c}=\begin{cases}
1/\sqrt{2c}, & N=3,\\
1/\sqrt{c}, & N\geq4.
\end{cases}
\]
Applying this inequality and the Cauchy-Schwarz inequality to the
above equation, we get
\begin{align*}
\left|\sum_{i,j,k=1}^{4}\left\langle \left[\omega_{ij},\omega_{jk}\right],\omega_{ki}\right\rangle \right| & \leq24\left|\left[\omega_{12},\omega_{23}\right]\right|\left|\omega_{31}\right|\\
 & \leq24\tilde{c}\left|\omega_{12}\right|\left|\omega_{23}\right|\left|\omega_{31}\right|.
\end{align*}
Using the inequality of arithmetic and geometric means, we have
\[
\left|\sum_{i,j,k=1}^{4}\left\langle \left[\omega_{ij},\omega_{jk}\right],\omega_{ki}\right\rangle \right|\leq\left(8\tilde{c}/\sqrt{3}\right)\left(\left|\omega_{12}\right|^{2}+\left|\omega_{23}\right|^{2}+\left|\omega_{31}\right|^{2}\right)^{3/2}.
\]
Using the fact that the form $\omega$ is self-dual and $\gamma=4\tilde{c}/\sqrt{6}$,
we conclude
\[
\left|\sum_{i,j,k=1}^{4}\left\langle \left[\omega_{ij},\omega_{jk}\right],\omega_{ki}\right\rangle \right|\leq\gamma\left|\omega\right|^{3}.
\]
This completes the first part of the proof.

Assume that $\omega\neq0$ and suppose that equality holds in the
above inequality.

First we prove Item (i). We claim that the $\omega_{ij}$ ($i\neq j$)
have the same norm and there exists a constant $t\neq0$ such that
$\left[\omega_{ij},\omega_{jk}\right]=t\omega_{ki}$ for cyclic indices
$i$, $j$, $k$ between $1$ and $3$. The proof is as follows. Since
equality holds in the above inequality, we see that the inequality
of arithmetic and geometric means (in the first part of the proof)
becomes an equality, so we see that the norms of $\omega_{12}$, $\omega_{23}$,
$\omega_{31}$ are equal. Using this and the fact that the form $\omega$
is self-dual, we have
\[
\left|\omega_{ij}\right|=\left(1/\sqrt{6}\right)\left|\omega\right|,\,\,\,\,\,i\neq j.
\]
Since equality holds in the above inequality, we see that the Cauchy-Schwarz
inequality and Bourguignon-Lawson's inequality (both in the first
part of the proof) become equalities, so we have
\[
\left|\left[\omega_{12},\omega_{23}\right]\right|=\tilde{c}\left|\omega_{12}\right|\left|\omega_{23}\right|,
\]
\[
\left[\omega_{12},\omega_{23}\right]\parallel\omega_{31}.
\]
Using the fact that $\left[\omega_{12},\omega_{23}\right]$ is a multiple
of $\omega_{31}$ and the above two equations, we have $\left[\omega_{12},\omega_{23}\right]=r\omega_{31}$,
where $r\in\left\{ \pm\left(\tilde{c}/\sqrt{6}\right)\left|\omega\right|\right\} $.
In the same way, we have $\left[\omega_{23},\omega_{31}\right]=s\omega_{12}$
and $\left[\omega_{31},\omega_{12}\right]=t\omega_{23}$, where $s,t\in\left\{ \pm\left(\tilde{c}/\sqrt{6}\right)\left|\omega\right|\right\} $.
Using the fact that the inner product of $so\left(N\right)$ is ad-invariant
and the $\omega_{ij}$ ($i\neq j$) have the same norm, we see that
$r$, $s$, $t$ are equal. This completes the proof of the claim.

Next we prove Item (ii). Suppose that $N=3$. By the proof of Item
(i), we have $\left|\left[\omega_{ij},\omega_{jk}\right]\right|=\tilde{c}\left|\omega_{ij}\right|\left|\omega_{jk}\right|$
for distinct indices $i$, $j$, $k$ between $1$ and $3$. Using
a result of Bourguignon-Lawson that describes this equality (Note
2.31 in \cite{BL1981}), we see that $\omega_{ij}\perp\omega_{jk}$
for distinct indices $i$, $j$, $k$ between $1$ and $3$. Using
this and the fact that the form $\omega$ is self-dual, we see that
the triple $\omega_{12}$, $\omega_{13}$, $\omega_{14}$ is an orthogonal
basis of $so\left(3\right)$.

Finally we prove Item (iii). Suppose that $N\geq4$. By the proof
of Item (i), we have $\left|\left[\omega_{12},\omega_{23}\right]\right|=\tilde{c}\left|\omega_{12}\right|\left|\omega_{23}\right|$.
Using a result of Bourguignon-Lawson that describes this equality
(Lemma 2.30 in \cite{BL1981}), we see that there exist constants
$a_{1}$, $a_{2}$, $a_{3}$ and a basis $\boldsymbol{i}$, $\boldsymbol{j}$,
$\boldsymbol{k}$ of $su\left(2\right)$ embedded into $so\left(N\right)$
(see Section 2) such that the pair $\omega_{12}$, $\omega_{23}$
is simultaneously orthogonally equivalent to a pair in the triple
$a_{1}\boldsymbol{i}$, $a_{2}\boldsymbol{j}$, $a_{3}\boldsymbol{k}$.
Using this, the fact that $\omega_{31}$ is a multiple of $\left[\omega_{12},\omega_{23}\right]$
(by the proof of Item (i)) and the fact that the form $\omega$ is
self-dual, we see (changing the constants and the basis if necessary)
that the triple $\omega_{12}$, $\omega_{13}$, $\omega_{14}$ is
simultaneously orthogonally equivalent to the triple $a_{1}\boldsymbol{i}$,
$a_{2}\boldsymbol{j}$, $a_{3}\boldsymbol{k}$.
\end{proof}

\section{Bochner formula}

In this section we prove a Bochner type formula based on Bourguignon-Lawson
\cite{BL1981} and Gursky-Kelleher-Streets \cite{GKS2018}. The description
of equality in the next result is very important in the proof of Theorem
\ref{thm:equality}.
\begin{lem}
\label{lem:bochner}Consider a four-dimensional Riemannian manifold
$X$ with scalar curvature $S$ and Weyl curvature $W$, a Riemannian
vector bundle $V$ on the manifold $X$ with structure group $G\subset O\left(N\right)$
and a Yang-Mills connection $A$ on the bundle $V$ with curvature
$F$, where $N\geq3$. Fix a constant $p>0$. Then
\begin{align*}
\left|F^{+}\right|^{p}\Delta\left|F^{+}\right|^{p}\geq & \left(1-1/\left(2p\right)\right)\left|\nabla\left|F^{+}\right|^{p}\right|^{2}+\left(p/3\right)S\left|F^{+}\right|^{2p}\\
 & -2p\sqrt{2/3}\left|W^{+}\right|\left|F^{+}\right|^{2p}-p\gamma\left|F^{+}\right|^{2p+1}.
\end{align*}
Also, if equality holds (in the above inequality) at a point $x$
in the manifold $X$ with $\left|F^{+}\right|\left(x\right)>0$, then
at this point
\[
\left|\sum_{i,j,k=1}^{4}\left\langle \left[F_{ij}^{+},F_{jk}^{+}\right],F_{ki}^{+}\right\rangle \right|=\gamma\left|F^{+}\right|^{3},
\]
\[
W^{+}=0.
\]

Here we use the inner products
\[
\left|F^{+}\right|^{2}=\sum_{1\leq i<j\leq4}\left|F_{ij}^{+}\right|^{2},\,\,\,\,\,\left\langle B,C\right\rangle _{so\left(V\right)}=-c\cdot tr\left(BC\right),\,\,\,\,\,c>0,
\]
and we denote $W^{+}=1/2\left(W+*W\right)$, $F^{+}=1/2\left(F+*F\right)$
and
\[
\gamma=\begin{cases}
4/\sqrt{12c}, & N=3,\\
4/\sqrt{6c}, & N\geq4.
\end{cases}
\]

The same statement is true replacing $F^{+}$ and $W^{+}$ by $F^{-}$
and $W^{-}$, respectively.
\end{lem}
\begin{proof}
We prove the result for $F^{+}$. The proof for $F^{-}$ is similar.

First we find an equation for the rough Laplacian of certain Hodge-harmonic
forms. Consider a self-dual Hodge-harmonic two-form $\omega$ on the
manifold $X$ with values in the bundle $so\left(V\right)$. Using
the fact that the form $\omega$ is Hodge-harmonic ($d\omega=0$ and
$d^{*}\omega=0$) and a Bochner type formula of Bourguignon-Lawson
(Theorem 3.10 in \cite{BL1981}), we have
\[
\left(\Delta\omega\right)_{ij}=-\sum_{k}R_{ik}\omega_{jk}+\sum_{k}R_{jk}\omega_{ik}-\sum_{k,l}R_{ijkl}\omega_{kl}+\sum_{k}\left[F_{ik},\omega_{jk}\right]-\sum_{k}\left[F_{jk},\omega_{ik}\right].
\]
Using the formula
\[
R_{ijkl}=W_{ijkl}+\left(1/2\right)\left(R_{ik}g_{jl}-R_{il}g_{jk}+g_{ik}R_{jl}-g_{il}R_{jk}\right)-\left(1/6\right)S\left(g_{ik}g_{jl}-g_{il}g_{jk}\right),
\]
and substituting the Riemann curvature $R_{ijkl}$ into the above
equation, we have
\[
\left(\Delta\omega\right)_{ij}=\left(1/3\right)S\omega_{ij}-\sum_{k,l}W_{ijkl}\omega_{kl}+\sum_{k}\left[F_{ik},\omega_{jk}\right]-\sum_{k}\left[F_{jk},\omega_{ik}\right].
\]
Taking the inner product with the form $\omega$ and renaming the
indices, we have
\[
\left\langle \Delta\omega,\omega\right\rangle =\left(1/3\right)S\left|\omega\right|^{2}-2\sum_{\left(i<j\right),\left(k<l\right)}W_{ijkl}\left\langle \omega_{ij},\omega_{kl}\right\rangle +\sum_{i,j,k}\left\langle \left[F_{ik},\omega_{jk}\right],\omega_{ij}\right\rangle .
\]
Seeing $W^{+}$ as a self-adjoint operator on $\Lambda_{+}^{2}$ and
using the fact that the form $\omega$ is self-dual ($*\omega=\omega$),
we conclude
\[
\left\langle \Delta\omega,\omega\right\rangle =\left(1/3\right)S\left|\omega\right|^{2}-2\sum_{\left(i<j\right),\left(k<l\right)}W_{ijkl}^{+}\left\langle \omega_{ij},\omega_{kl}\right\rangle +\sum_{i,j,k}\left\langle \left[F_{ik}^{+},\omega_{jk}\right],\omega_{ij}\right\rangle .
\]

Next we substitute $F^{+}$ into this equation. Using the Bianchi
identity and the fact that the connection $A$ is Yang-Mills, we see
that the curvature $F$ is Hodge-harmonic, so we deduce that $F^{+}$
is Hodge-harmonic ($dF^{+}=0$ and $d^{*}F^{+}=0$). We also see that
$F^{+}$ is self-dual ($*F^{+}=F^{+}$). Substituting $\omega=F^{+}$
into the above equation and renaming the indices in the last term,
we conclude
\[
\left\langle \Delta F^{+},F^{+}\right\rangle =\left(1/3\right)S\left|F^{+}\right|^{2}-2\sum_{\left(i<j\right),\left(k<l\right)}W_{ijkl}^{+}\left\langle F_{ij}^{+},F_{kl}^{+}\right\rangle +\sum_{i,j,k}\left\langle \left[F_{ij}^{+},F_{jk}^{+}\right],F_{ki}^{+}\right\rangle .
\]

Next we apply inequalities to this equation. Using the fact that $W^{+}$
is a trace-free self-adjoint operator on $\Lambda_{+}^{2}$ and the
inequality $3c^{2}\leq2\left(a^{2}+b^{2}+c^{2}\right)$ for $a+b+c=0$,
we have
\begin{align*}
\sum_{\left(i<j\right),\left(k<l\right)}W_{ijkl}^{+}\left\langle F_{ij}^{+},F_{kl}^{+}\right\rangle  & \leq\lambda_{max}\left(W^{+}\right)\left|F^{+}\right|^{2}\\
 & \leq\sqrt{2/3}\left|W^{+}\right|\left|F^{+}\right|^{2}.
\end{align*}
Using Lemma \ref{lem:linearalgebra}, we have
\[
\left|\sum_{i,j,k}\left\langle \left[F_{ij}^{+},F_{jk}^{+}\right],F_{ki}^{+}\right\rangle \right|\leq\gamma\left|F^{+}\right|^{3}.
\]
Using a Kato inequality of Gursky-Kelleher-Streets (Proposition 2.8
in \cite{GKS2018}), we have
\[
\left|\nabla F^{+}\right|^{2}\geq\left(3/2\right)\left|\nabla\left|F^{+}\right|\right|^{2}.
\]
Applying these three inequalities to the above equation, we get
\begin{align*}
\left|F^{+}\right|\Delta\left|F^{+}\right| & =\left\langle \Delta F^{+},F^{+}\right\rangle +\left|\nabla F^{+}\right|^{2}-\left|\nabla\left|F^{+}\right|\right|^{2}\\
 & \geq\left(1/2\right)\left|\nabla\left|F^{+}\right|\right|^{2}+\left(1/3\right)S\left|F^{+}\right|^{2}-2\sqrt{2/3}\left|W^{+}\right|\left|F^{+}\right|^{2}-\gamma\left|F^{+}\right|^{3}.
\end{align*}
By a simple calculation, we conclude
\begin{align*}
\left|F^{+}\right|^{p}\Delta\left|F^{+}\right|^{p}\geq & \left(1-1/\left(2p\right)\right)\left|\nabla\left|F^{+}\right|^{p}\right|^{2}+\left(p/3\right)S\left|F^{+}\right|^{2p}\\
 & -2p\sqrt{2/3}\left|W^{+}\right|\left|F^{+}\right|^{2p}-p\gamma\left|F^{+}\right|^{2p+1}.
\end{align*}

Finally suppose that equality holds in the above inequality at a point
$x$ in the manifold $X$ with $\left|F^{+}\right|\left(x\right)>0$.
In this case, we see that all the above inequalities become equalities
at this point, so at the point $x$ we have
\[
\left|\sum_{i,j,k}\left\langle \left[F_{ij}^{+},F_{jk}^{+}\right],F_{ki}^{+}\right\rangle \right|=\gamma\left|F^{+}\right|^{3},
\]
\[
\lambda_{max}\left(W^{+}\right)\left|F^{+}\right|^{2}=\sum_{\left(i<j\right),\left(k<l\right)}W_{ijkl}^{+}\left\langle F_{ij}^{+},F_{kl}^{+}\right\rangle .
\]
Using the first equation and Lemma \ref{lem:linearalgebra}, we see
that at the point $x$ the $F_{ij}^{+}$ ($i\neq j$) have the same
norm. Diagonalizing the trace-free self-adjoint operator $W^{+}$
in the second equation and using the fact that the $F_{ij}^{+}$ ($i\neq j$)
have the same norm, we see that at the point $x$ we have
\[
\lambda_{max}\left(W^{+}\right)\left|F^{+}\right|^{2}=0,
\]
so we get $\lambda_{max}\left(W^{+}\right)\left(x\right)=0$, which
implies $W^{+}\left(x\right)=0$.
\end{proof}

\section{Proof of Theorem \ref{thm:inequality} and Theorem \ref{thm:equality}}

In this section we prove Theorem \ref{thm:inequality} and Theorem
\ref{thm:equality} and we apply these results to the sphere $S^{4}$,
the Euclidean space $R^{4}$ and the cylinder $S^{3}\times R$. The
key ideas in the proof of Theorem \ref{thm:equality} are to show
that the function $I$ is identically zero and to use the description
of equality in Lemma \ref{lem:linearalgebra} and Lemma \ref{lem:bochner}.

First we prove Theorem \ref{thm:inequality}.
\begin{proof}
Take the function $I$ on the manifold $X$ given by
\[
I=\left|F^{+}\right|^{1/2}\Delta\left|F^{+}\right|^{1/2}-c_{1}S\left|F^{+}\right|+c_{2}\left|W^{+}\right|\left|F^{+}\right|+c_{3}\left|F^{+}\right|^{2},
\]
where $c_{1}=1/6$, $c_{2}=\sqrt{2/3}$ and $c_{3}=\gamma/2$. Using
Lemma \ref{lem:bochner}, we see that this function is nonnegative,
that is
\[
I\geq0.
\]
Take the cutoff function $\phi$ on the manifold $X$ given by
\[
\phi=\begin{cases}
1 & B_{R},\\
2-r/R & B_{2R}\setminus B_{R},\\
0 & X\setminus B_{2R},
\end{cases}
\]
where $r=dist_{X}\left(\cdot,x_{0}\right)$ and $B_{R}=\left\{ r<R\right\} $.
Multiplying the function $I$ by the function $\phi^{2}$ and integrating
by parts, we have
\begin{align*}
\int I\phi^{2}= & -\int\left|\nabla\left|F^{+}\right|^{1/2}\right|^{2}\phi^{2}-2\int\left|F^{+}\right|^{1/2}\phi\left\langle \nabla\left|F^{+}\right|^{1/2},\nabla\phi\right\rangle \\
 & -c_{1}\int S\left|F^{+}\right|\phi^{2}+c_{2}\int\left|W^{+}\right|\left|F^{+}\right|\phi^{2}+c_{3}\int\left|F^{+}\right|^{2}\phi^{2}.
\end{align*}
On the other hand, substituting the compactly supported function $\left|F^{+}\right|^{1/2}\phi$
into the inequality determined by the Yamabe constant $C_{Y}$, we
have
\begin{align*}
c_{4}\left(\int\left|F^{+}\right|^{2}\phi^{4}\right)^{1/2}\leq & \int\left|\nabla\left|F^{+}\right|^{1/2}\right|^{2}\phi^{2}+2\int\left|F^{+}\right|^{1/2}\phi\left\langle \nabla\left|F^{+}\right|^{1/2},\nabla\phi\right\rangle \\
 & +\int\left|F^{+}\right|\left|\nabla\phi\right|^{2}+c_{1}\int S\left|F^{+}\right|\phi^{2},
\end{align*}
where $c_{4}=C_{Y}/6$. Summing this inequality and the above equation,
we get
\begin{align*}
\int I\phi^{2}+c_{4}\left(\int\left|F^{+}\right|^{2}\phi^{4}\right)^{1/2} & \leq c_{2}\int\left|W^{+}\right|\left|F^{+}\right|\phi^{2}+c_{3}\int\left|F^{+}\right|^{2}\phi^{2}\\
 & +\int\left|F^{+}\right|\left|\nabla\phi\right|^{2}.
\end{align*}
We can assume that $F^{+}$ and $W^{+}$ are in $L^{2}$, otherwise
the conclusion of the theorem is trivial. Using Holder's inequality,
we have
\[
\int I\phi^{2}+\left(c_{4}-c_{2}\left|W^{+}\right|_{L^{2}}-c_{3}\left|F^{+}\right|_{L^{2}}\right)\left(\int\left|F^{+}\right|^{2}\phi^{4}\right)^{1/2}\leq\int\left|F^{+}\right|\left|\nabla\phi\right|^{2}.
\]
Using Holder's inequality and the fact that volume growth is at most
Euclidean ($vol\left(B_{R}\right)\leq const\cdot R^{4}$), we have
\begin{align*}
 & \int I\phi^{2}+\left(c_{4}-c_{2}\left|W^{+}\right|_{L^{2}}-c_{3}\left|F^{+}\right|_{L^{2}}\right)\left(\int\left|F^{+}\right|^{2}\phi^{4}\right)^{1/2}\\
 & \leq R^{-2}\left(\int_{B_{2R}\setminus B_{R}}\left|F^{+}\right|^{2}\right)^{1/2}\left(vol\left(B_{2R}\right)\right)^{1/2}\\
 & \leq const\left(\int_{B_{2R}\setminus B_{R}}\left|F^{+}\right|^{2}\right)^{1/2}.
\end{align*}
Taking the limit as the radius $R\to\infty$ and using the fact that
the function $I$ is nonnegative, we get
\[
\left(c_{4}-c_{2}\left|W^{+}\right|_{L^{2}}-c_{3}\left|F^{+}\right|_{L^{2}}\right)\left(\int\left|F^{+}\right|^{2}\right)^{1/2}\leq0.
\]
Since $F^{+}$ is not identically zero, we conclude
\[
c_{4}\leq c_{2}\left|W^{+}\right|_{L^{2}}+c_{3}\left|F^{+}\right|_{L^{2}}.
\]
\end{proof}
Next we prove Theorem \ref{thm:equality}.
\begin{proof}
By the proof of Theorem \ref{thm:inequality}, we have
\[
\int I\phi^{2}+\left(c_{4}-c_{3}\left|F^{+}\right|_{L^{2}}\right)\left(\int\left|F^{+}\right|^{2}\phi^{4}\right)^{1/2}\leq const\left(\int_{B_{2R}\setminus B_{R}}\left|F^{+}\right|^{2}\right)^{1/2}.
\]
Since $c_{4}-c_{3}\left|F^{+}\right|_{L^{2}}=0$, we have
\[
\int I\phi^{2}\leq const\left(\int_{B_{2R}\setminus B_{R}}\left|F^{+}\right|^{2}\right)^{1/2}.
\]
Taking the limit as the radius $R\to\infty$ and using the fact that
the function $I$ is nonnegative (by the proof of Theorem \ref{thm:inequality}),
we see that the function $I$ is identically zero, that is
\[
\left|F^{+}\right|^{1/2}\Delta\left|F^{+}\right|^{1/2}=c_{1}S\left|F^{+}\right|-c_{3}\left|F^{+}\right|^{2}.
\]
Using Lemma \ref{lem:bochner}, we have
\[
\left|\sum_{i,j,k=1}^{4}\left\langle \left[F_{ij}^{+},F_{jk}^{+}\right],F_{ki}^{+}\right\rangle \right|=\gamma\left|F^{+}\right|^{3}.
\]
Take normal coordinates at a point in the manifold $X$. Using Lemma
\ref{lem:linearalgebra}, we conclude that at this point the $F_{ij}^{+}$
($i\neq j$) have the same norm and the triple $F_{12}^{+}$, $F_{13}^{+}$,
$F_{14}^{+}$ is simultaneously orthogonally equivalent to a triple
$a_{1}\boldsymbol{i}$, $a_{2}\boldsymbol{j}$, $a_{3}\boldsymbol{k}$
of multiples of a basis of $su\left(2\right)$ embedded into $so\left(N\right)$.
\end{proof}
Finally, applying Theorem \ref{thm:inequality} and Theorem \ref{thm:equality}
to the sphere $S^{4}$, the Euclidean space $R^{4}$ and the cylinder
$S^{3}\times R$, we get the following result:
\begin{cor}
Consider the sphere $S^{4}$, the Euclidean space $R^{4}$ and the
cylinder $S^{3}\times R$ with the product metric. Consider a Yang-Mills
connection $A$ with curvature $F$ and structure group $G\subset O\left(N\right)$
on any of these manifolds, where $N\geq4$. If $F^{+}\neq0$, then
\[
\left|F^{+}\right|_{L^{2}}\geq4\omega_{4}^{1/2}\gamma^{-1}.
\]
Here $\omega_{4}=vol\left(S^{4}\right)$. Also, if equality holds
(in the above inequality), then at each point in normal coordinates
the triple $F_{12}^{+}$, $F_{13}^{+}$, $F_{14}^{+}$ is simultaneously
orthogonally equivalent to a triple $a_{1}\boldsymbol{i}$, $a_{2}\boldsymbol{j}$,
$a_{3}\boldsymbol{k}$ of multiples of a basis of $su\left(2\right)$
embedded into $so\left(N\right)$, and the absolute values of $a_{1}$,
$a_{2}$, $a_{3}$ are equal. The same statement is true replacing
$F^{+}$ by $F^{-}$.
\end{cor}
\begin{proof}
For the sphere $S^{4}$, the Euclidean space $R^{4}$ and the cylinder
$S^{3}\times R$ we see that the Weyl curvature is zero ($W=0$) and
the volume growth is at most Euclidean ($vol\left(B_{R}\right)\leq const\cdot R^{4}$).

For the sphere $S^{4}$ and the Euclidean space $R^{4}$, using results
of Aubin \cite{A1976a}-\cite{A1976b} and Talenti \cite{T1976},
we see that the Yamabe constant is $C_{Y}=12\omega_{4}^{1/2}$.

For the cylinder $S^{3}\times R$, using a result of Ammann-Dahl-Humbert
(Lemma 3.7 with $k=0$ and Equation 4 in \cite{ADH2013}), we see
that the Yamabe constant is also $C_{Y}=12\omega_{4}^{1/2}$.
\end{proof}
We would like to thank Daniel Fadel, Almir Santos and Paul Feehan
for the support.

\lyxaddress{Departamento de Matemática, Universidade Federal do Espírito Santo,
Vitória, ES, Brazil. Email address: matheus.vieira@ufes.br}
\end{document}